\documentclass[10pt]{amsart}
\usepackage{amssymb} 

\def\Corresponding author{$^{*}$\protect\footnotetext{$^{*}$ \lowercase{Corresponding author.}}}
\newtheorem{theorem}{Theorem}[section]

\newtheorem{lemma}[theorem]{Lemma}
\newtheorem{proposition}[theorem]{Proposition}

\theoremstyle{definition}
\newtheorem{definition}[theorem]{Definition}

\numberwithin{equation}{section}

\begin{document}
\title[On equality of inner and absolute central automorphisms]
{On equality of inner and absolute central automorphisms}
\author[ Kaboutari. Nasrabadi ]{Z. Kaboutari Farimani\\M. M. Nasrabadi \Corresponding author\\}%
\address{Department of Mathematics, University of Birjand, Birjand, Iran.}%
\email{z$_ {_{-}}$kaboutari@birjand.ac.ir, kaboutarizf@gmail.com, mnasrabadi@birjand.ac.ir}%
\thanks{}
\maketitle
\begin{abstract} Let $G$ be a finite $p$-group and let $\mathrm{Aut}_l(G)$ be the group of absolute central automorphisms of $G$. We give necessary and sufficient conditions on $G$ such that $\mathrm{Aut}_l(G)=\mathrm{Inn}(G)$.
\normalsize \noindent
\end{abstract}
\section{\bf Introduction } An automorphism $\alpha$ of $G$ is called
central if $x^{-1}\alpha(x)\in Z(G)$ for each $x\in
G$. The central automorphisms of $G$, denoted by
$\mathrm{Aut}_c(G)$, fix $G^\prime$ elementwise and form a normal
subgroup of the full automorphism group of $G$. The properties of $\mathrm{Aut}_c(G)$ have been well studied. See \cite{Adney,Cur,JAM} for example. 

Hegarty in  \cite{Heg} generalized the concept of centre into absolute centre. The \textit{absolute centre} of a group $G$, denoted by $L(G)$, is the subgroup consisting of all those elements that are fixed under all automorphisms of $G$.
Also he introduced the absolute central automorphisms. An automorphism $\beta$ of $G$ is called an
absolute central automorphism if $x^{-1}\beta(x)\in L(G)$ for each $x\in
G$. We denote the set of all absolute central automorphisms of $G$ by $\mathrm{Aut}_l(G)$. Notice that  $\mathrm{Aut}_l(G)$ is a normal subgroup of $\mathrm{Aut}(G)$ contained in $\mathrm{Aut}_c(G)$.

In \cite{Cur} authors gave necessary and sufficient conditions on a $p$-group $G$ such that $\mathrm{Aut}_c(G)=\mathrm{Inn}(G)$. In \cite{Nasr} we gave conditions on a finite autonilpotent $p$-group $G$ of class $2$ such that $\mathrm{Aut}_l(G)=\mathrm{Inn}(G)$. In this paper we intend to give necessary and sufficient conditions on a non-abelian $p$-group $G$ in which $\mathrm{Aut}_l(G)$ and $\mathrm{Inn}(G)$ coincide.

Throughout this paper all groups are assumed to be finite and $p$ denotes a prime number.  Also if $G$ is a group, then $\mathrm{exp}(G)$, $\mathrm{Hom}(G,H)$ and $G^{p^n}$ stand for the exponent, the group of homomorphisms of $G$ into an abelian group $H$ and the subgroup generated by all $p^nth$ powers of elements of $G$, respectively. 
 Let $M$ and $N$ be two normal subgroups of $G$. We denote the subgroup of $\mathrm{Aut}(G)$ consisting of all automorpisms centralizing $G/M$ by
 $\mathrm{Aut}^M(G)$ and the subgroup of $\mathrm{Aut}(G)$ consisting of all automorphisms which act trivially on $N$ by $\mathrm{Aut}_N(G)$. Also, we consider
 $\mathrm{Aut}_{N}^M(G)=\mathrm{Aut}^M(G)\cap\mbox{Aut}_N(G)$.
\section{Preliminary results }
In this section, we give some results that will be used in the proof of the main results.
\begin{lemma}\cite[Lemma 2.3]{Nasr}
Suppose $H$ is an abelian $p$-group of exponent $p^c$, and $K$ is a cyclic group of order divisible by $p^c$. Then $\mathrm{Hom}(H,K)$ is isomorphic to $H$.
\end{lemma}
Let $G$ be a group. Then the \textit{autocommutator} of an element $g\in G$ and automorphism $\alpha\in \mathrm{Aut}(G)$ is defined as $[g,\alpha]=g^{-1}g^{\alpha}$.
\begin{definition}
The \textit{absolute centre} of a group $G$, denoted by $L(G)$, is defined as 
$$L(G)=\{g\in G~|~[g,\alpha]=1,~ \forall \alpha\in \mathrm{Aut}(G)\}$$

Clearly,  $L(G)$ is a central characteristic subgroup of $G$.
\end{definition}
\begin{definition}
An automorphism $\alpha$ of $G$ is called \textit{absolute central}, if $g^{-1}\alpha(g)\in L(G)$ for each $g\in G$. The set of all absolute central automorphisms of $G$ is denoted by $\mathrm{Aut}_l(G)$.

 Clearly, $\mathrm{Aut}_l(G)$ is an abelian normal subgroup of $\mathrm{Aut}(G)$.
\end{definition}
\begin{proposition}
\cite[Proposition 1]{Safa} Let $G$ be a group. Then \[\mathrm{Aut}_l(G)\cong \mathrm{Hom}(G/L(G),L(G)).\]
\end{proposition}
\begin{lemma}
Let $G$ be a group. Then \[\mathrm{Aut}_{Z(G)}^{L(G)}(G)\cong\mathrm{Hom}(G/Z(G),L(G)).\]
\end{lemma}
\begin{proof} 
Consider the map $\bar{\alpha}:G/Z(G)\longrightarrow L(G)$ defined by $\bar{\alpha}(gZ(G))=g^{-1}\alpha(g)$ for all $g\in G$ and each $\alpha\in\mbox{Aut}_{Z(G)}^{L(G)}(G)$. Clearly, $\bar{\alpha}$ is a well-defined homomorphism of $G/Z(G)$ into $L(G)$. Now, a simple verification shows that the map $\phi:\mathrm{Aut}_{Z(G)}^{L(G)}(G)\longrightarrow \mathrm{Hom}(G/Z(G),L(G))$ defined by $\phi(\alpha)=\bar{\alpha}$, for any $\alpha\in\mbox{Aut}_{Z(G)}^{L(G)}(G)$, is an isomorphism.
\end{proof}
\begin{lemma}
Let $G$ be a  $p$-group of class $2$ such that $\mathrm{exp}(G/Z(G))\leq\mathrm{exp}(L(G))$. Then $$|\mathrm{Hom}(G/Z(G),L(G))|\geq |G/Z(G)|p^{r(s-1)}$$
where $r=\mathrm{rank}(G/Z(G))$ and $s=\mathrm{rank}(L(G))$.
\end{lemma}
\begin{proof}
Let $\mathrm{exp}(L(G))=p^n$ and $\mathrm{exp}(G/Z(G))=p^c$. Clearly, $r\geq 2$. Now let
$$L(G)=C_{p^{n}}\times C_{p^{\gamma_{2}}}\times  \dots \times C_{p^{\gamma_{s}}} $$ and
$$G/Z(G)=C_{p^c}\times C_{p^{\beta_{2}}}\times \dots \times C_{p^{\beta_{r}}}$$
where $n\geq \gamma_{2}\geq \dots \geq \gamma_{s}> 0$ and $c\geq \beta_{2}\geq \dots \geq \beta_{r}> 0$.
Since $G$ is a $p$-group and $\mathrm{exp}(G/Z(G))\leq\mbox{exp}(L(G))$,
 $\mathrm{exp}(G/Z(G))$ divides $\mathrm{exp}(L(G))$. The rest of proof is similar to that of \cite[Lemma 2.8]{Nasr} so we omit the details.
\end{proof}
Now we establish a lower bound for $|\mathrm{Aut}_l(G)|$.
\begin{lemma}
Let $G$ be a  $p$-group of class $2$ such that $\mathrm{exp}(G/Z(G))\leq\mathrm{exp}(L(G))$. Then $|\mathrm{Aut}_l(G)|\geq |G/Z(G)| p^{r(s-1)}$,
where $r$ and $s$ are as defined before.
\end{lemma}
\begin{proof}
We have
$$|\mathrm{Aut}_{Z(G)}^{L(G)}(G)|=|\mathrm{Hom}(G/Z(G),L(G))|\leq |\mathrm{Hom}(G/L(G),L(G))|=|\mathrm{Aut}_l(G)|.$$
So by Lemma $2.6$, $|\mathrm{Aut}_l(G)|\geq|G/Z(G)|p^{r(s-1)}$. 
\end{proof}
\begin{lemma}\cite[Lemma 2.7]{Nasr}
Let $G$ be a group. Then $G/L(G)$ is abelian if and only if $\mathrm{Inn}(G)\leq \mathrm{Aut}_l(G)$.
\end{lemma}

\section{ Main results}
In this section, we obtain some properties of the group $G$ when $\mathrm{Aut}_l(G)=\mathrm{Inn}(G)$, and then give necessary and sufficient conditions under which $\mathrm{Aut}_l(G)=\mathrm{Inn}(G)$.
\begin{lemma}
Let $G$ be a non-abelian $p$-group. If $\mathrm{Aut}_l(G)=\mathrm{Inn}(G)$, then $L(G)$ is cyclic.
\end{lemma}
\begin{proof}
Let $\mathrm{Aut}_l(G)=\mathrm{Inn}(G)$. Thus $\mathrm{Hom}(G/L(G),L(G))\cong G/Z(G)$. Hence $\mathrm{exp}(G/Z(G))\leq\mathrm{exp}(L(G))$.
Also, $\mathrm{Inn}(G)$ is abelian so that $G$ is nilpotent of class $2$. Now, by Lemma $2.7$,  we have $p^{r(s-1)}\leq \frac{|\mathrm{Aut}_l(G)|}{|\mathrm{Inn}(G)|}$, where
 $r=\mathrm{rank}(G/Z(G))$ and $s=\mathrm{rank}(L(G))$. Thus $p^{r(s-1)}\leq 1$. Since $r\geq 2$ we must have $s=1$, that is, $L(G)$ is cyclic.
\end{proof}
\begin{proposition}
Let $G$ be a non-abelian $p$-group ($p$ odd) such that $\mathrm{Aut}_l(G)=\mathrm{Inn}(G)$. Then $\exp(G/Z(G))=\exp(L(G))$.
\end{proposition}
\begin{proof}
Suppose $p$ is an odd prime number and $\mathrm{exp}(L(G))=p^n$. Consider the map $\theta:G\longrightarrow G$ defined by $\theta(x)=x^{1+p^{n-1}}$ for all $x\in G$.
Clearly, $\theta$ is an automorphism when  $\mathrm{exp}(G/Z(G))<\mathrm{exp}(L(G))$. In this case, $g^{p^{n-1}}=1$ for all $g$ in $L(G)$. Hence $\mathrm{exp}(L(G))<p^n$, which is a contradiction. Thus $\mathrm{exp}(G/Z(G))=\mathrm{exp}(L(G))$.
\end{proof}
\begin{theorem}
Let $G$ be a non-abelian $p$-group. Then $\mathrm{Inn}(G)=\mathrm{Aut}_{Z(G)}^{L(G)}(G)$ if and only if $G'\leq L(G)$ and $L(G)$ is cyclic.
\end{theorem}
\begin{proof}
Suppose $G'\leq L(G)$ and $L(G)$ is cyclic. Since $G'\leq L(G)$, $G$ is nilpotent of class $2$ and $\mathrm{exp}(G')=\mathrm{exp}(G/Z(G))$. Hence $\mathrm{exp}(G/Z(G))$ divides $\mathrm{exp}(L(G))$ and by Lemma $2.1$, $\mathrm{Hom}(G/Z(G),L(G))\cong G/Z(G)$. Therefore by Lemma $2.5$, $\mathrm{Aut}_{Z(G)}^{L(G)}(G)\cong\mbox{Inn}(G)$. On the other hand, by Lemma $2.8$, we have $\mathrm{Inn}(G)\leq \mathrm{Aut}_l(G)$. Now since $\mathrm{Inn}(G)$ fixes the centre element-wise, we conclude that
$\mathrm{Inn}(G)\leq \mathrm{Aut}_{Z(G)}^{L(G)}(G)$. Hence $\mathrm{Inn}(G)=\mathrm{Aut}_{Z(G)}^{L(G)}(G)$.

To prove the converse, assume that $\mathrm{Inn}(G)=\mathrm{Aut}_{Z(G)}^{L(G)}(G)$. Let $x\in G$ and $\theta_x$ be the inner automorphism of $G$ induced by $x$. Then, 
$g^{-1}\theta_x(g)=[g,x]\in L(G)$ for all $g$ in $G$, and consequently $G'\leq L(G)$. Hence, 
\begin{eqnarray*}
|G/Z(G)|=|\mathrm{Inn}(G)|&=&|\mathrm{Aut}_{Z(G)}^{L(G)}(G)|\\
&=&|\mathrm{Hom}(G/Z(G),L(G))|\geq |G/Z(G)|p^{r(s-1)}
\end{eqnarray*}
where $r=\mathrm{rank}(G/Z(G))$ and $s=\mathrm{rank}(L(G))$. Thus $p^{r(s-1)}=1$. Since $r\geq 2$, we must have $s=1$ so that $L(G)$ is cyclic.
\end{proof}
\begin{theorem}
Let $G$ be a non-abelian $p$-group. Then $\mathrm{Aut}_l(G)=\mathrm{Inn}(G)$ if and only if $G'\leq L(G)$, $L(G)$ is cyclic and $Z(G)=L(G)G^{p^n}$, where $p^n=\mathrm{exp}(L(G))$.
\end{theorem}
\begin{proof}
Suppose first that $G'\leq L(G)$, $L(G)$ is cyclic and $Z(G)=L(G)G^{p^n}$. By Theorem $3.3$, 
$\mathrm{Inn}(G)=\mathrm{Aut}_{Z(G)}^{L(G)}(G)$. Now suppose $\alpha\in \mathrm{Aut}_l(G)$. Let $g\in G$ and $l\in L(G)$ be such that $\alpha(g)=gl$. Then $\alpha(g^{p^n})=g^{p^n}l^{p^n}$. Since $\mathrm{exp}(L(G))=p^n$, we have $l^{p^n}=1$ and so $\alpha(g^{p^n})=g^{p^n}$. Hence $\alpha$ acts trivially on $Z(G)$, that is, $\alpha\in \mathrm{Aut}_{Z(G)}^{L(G)}(G)$. Thus $\mathrm{Aut}_l(G)\leq\mathrm{Aut}_{Z(G)}^{L(G)}(G)$ so that $\mathrm{Aut}_l(G)=\mathrm{Aut}_{Z(G)}^{L(G)}(G)=\mathrm{Inn}(G)$, as required.

Conversely, suppose $\mathrm{Aut}_l(G)=\mathrm{Inn}(G)$. By Lemma $3.1$, $L(G)$ is cyclic and by Lemma $2.8$, $G'\leq L(G)$. Now we have $\mathrm{exp}(G')\leq\mathrm{exp}(L(G))=p^n$. Hence for all $a,b\in G$, $[a^{p^n},b]=1$. This means $a^{p^n}\in Z(G)$, for all $a\in G$. Therefore $G^{p^n}\leq Z(G)$ whence $L(G)G^{p^n}\leq Z(G)$. Since $L(G)\leq Z(G)$, it follows that
$$\mathrm{Hom}(G/Z(G),L(G))\hookrightarrow \mathrm{Hom}(G/L(G)G^{p^n},L(G))\hookrightarrow \mathrm{Hom}(G/L(G),L(G)).$$
As \[\mathrm{Inn}(G)\cong G/Z(G)\cong \mathrm{Hom}(G/Z(G),L(G))\] and \[\mathrm{Hom}(G/L(G),L(G))\cong \mathrm{Aut}_l(G)=\mathrm{Inn}(G),\] we observe that \[\mathrm{Hom}(G/Z(G),L(G))\cong \mathrm{Hom}(G/L(G)G^{p^n},L(G)),\] hence $|G/Z(G)|=|G/L(G)G^{p^n}|$ from which it follows that $Z(G)=L(G)G^{p^n}$.The proof is complete.
\end{proof}

\section*{Acknowledgments}
The authors are grateful to Dr. M. Farrokhi Derakhshandeh Ghouchan for his valuable suggestions and help in carrying out this work. 

\end{document}